\documentclass[10pt,reqno]{article}

\usepackage{amsmath}
\usepackage{latexsym}
\usepackage{amsmath,amsthm}
\usepackage{amsfonts}
\usepackage{amssymb}
\usepackage{amsfonts,euscript}
\usepackage{graphicx}

\DeclareGraphicsExtensions{.eps,epsfig,.bmp,.jpg,.pdf,.mps,.png,.gif}
\newtheorem{theorem}{\bf Theorem}[section]

\newtheorem{corollary}{\bf Corollary}[section]
\newtheorem{lemma}{\bf Lemma}[section]
\newtheorem{remark}{\bf Remark}[section]

\newcommand{\beq}{\begin{equation}}
\newcommand{\eeq}{\end{equation}}
\newcommand{\beqn}{\begin{eqnarray}}
\newcommand{\eeqn}{\end{eqnarray}}
\newcommand{\bear}{\begin{array}}
\newcommand{\eear}{\end{array}}
\newcommand{\beit}{\begin{itemize}}
\newcommand{\eeit}{\end{itemize}}
\newcommand{\beqno}{\begin{eqnarray*}}
\newcommand{\eeqno}{\end{eqnarray*}}

\allowdisplaybreaks

\title{Global solutions for  random vorticity equations perturbed by gradient dependent noise, in two and three dimensions}
\date{}

\numberwithin{equation}{section}

\begin{document}
\maketitle

\centerline{\scshape Ionu\c t Munteanu}
\medskip
{\footnotesize
  \centerline{Alexandru Ioan Cuza University of Ia\c si, Department of Mathematics }
   \centerline{Blvd. Carol I, no.11, 700506-Ia\c si, Romania}

   }
   {\footnotesize
  \centerline{Octav Mayer Institute of Mathematics, Romanian Academy }
   \centerline{Blvd. Carol I, no.8, 700505-Ia\c si, Romania}
\centerline{e-mail: ionut.munteanu@uaic.ro}
   }
\medskip

\centerline{\scshape Michael R$\ddot{o}$ckner}
\medskip
{\footnotesize
 \centerline{Fakultat fur Mathematik, Universitat Bielefeld  
}
     \centerline{D-33501 Bielefeld, Germany}
\centerline{e-mail: roeckner@math.uni-bielefeld.de}} 
\medskip
\noindent\begin{abstract}The aim of this work is to prove an existence and uniqueness result  of Kato-Fujita type  for  the Navier-Stokes equations, in vorticity form, in $2-D$ and $3-D$, perturbed by a gradient type multiplicative Gaussian noise (for sufficiently small initial vorticity). These equations are considered in order to model hydrodynamic turbulence. The approach was motivated by a recent result by V. Barbu and the second named author in \cite{b1}, that treats the stochastic $3D$-Navier-Stokes equations, in vorticity form, perturbed by linear multiplicative Gaussian noise. More precisely,  the equation is transformed to a random nonlinear parabolic equation, as in \cite{b1}, but the transformation is different and adapted to our gradient type noise. Then global unique existence results are proved for the transformed equation, while for the original stochastic Navier-Stokes equations,   existence of a solution adapted to the Brownian filtration is obtained  up to some stopping time.
\end{abstract}
\noindent \textbf{Keywords:} stochastic Navier-Stokes equation, turbulence, vorticity, Biot-Savart operator, gradient-type noise.

\noindent\textbf{MSC:}60H15, 35Q30, 76F20, 76N10.
\section{Introduction}
One of the main important features concerning the Navier-Stokes equation is its relation to the phenomenon of hydrodynamic turbulence, that  is often assumed to be caused by random background movements. That is why a randomly forced Navier-Stokes equation may be considered to model  this. In this direction, we recall the pioneering work of Bensoussan and Temam \cite{t} concerning the analytical study of a Navier-Stokes equation driven by a white noise type random force; followed later by numerous developments and extensions by many authors (see \cite{brez2, fl, mik, brez3} and the references therein).  We emphasize the approach in \cite{mik,mik2} that involves gradient dependent noise in order to model turbulence. In this light, we consider the following Navier-Stokes equation in dimension $d=2,3,$ perturbed by gradient dependent noise
\begin{equation}\label{e1}\left\{ \begin{array}{l}\displaystyle dX-\Delta Xdt+(X\cdot\nabla)Xdt=\sum_{i=1}^N A_i(X) d\beta_i(t)+\nabla\pi dt  \text{ on } (0,\infty)\times \mathbb{R}^d,\\
\nabla\cdot X =0 \text{ on } (0,\infty)\times \mathbb{R}^d,\\
X(0)=x \text{ in }\left(L^p(\mathbb{R}^d\right)^d,\end{array} \right.\ \end{equation}where $ x:\Omega\rightarrow \mathbb{R}^d$ is a random variable; $\pi$ denotes the pressure; $\left\{\beta_i\right\}_{i=1}^N$ is a system of independent Brownian motions on a probability space $\left(\Omega, \mathcal{F},\mathbb{P}\right)$ with normal filtration $(\mathcal{F}_t)_{t\geq0}$, $x$ is $\mathcal{F}_0-$adapted, and $A_i$ are certain operators, linear in the gradient of the solution,  specified below. 

Our aim in this paper is to study (\ref{e1}) by writing it in vorticity form (i.e., apply  the $curl$ operator to it) and by  transforming it into the following random partial differential equation
\begin{equation}\label{ho10}\begin{aligned}\frac{d y}{dt}=&\Delta y(t)+\Gamma^{-1}(t)[K(\Gamma(t)y(t))\cdot \nabla](\Gamma(t)y(t)),\ t>0;\ y(0)=U_0=curl \ x.\end{aligned}\end{equation}
where $\Gamma(t)$ solves \eqref{edi1} below and $K$ is the Biot-Savart operator. We analyze \eqref{ho10} with $\left\{\beta_i(\omega)\right\}_{i=1}^N$ for a.e. fixed $\omega$. In particular, we are going to prove a Kato-Fujita type result (see \cite{kato}), i.e., we prove that for small enough initial condition there exists a globally in time unique solution of (\ref{ho10}). The smallness of the initial conditions depends however on $\omega$ (see (\ref{e1100}) below). To this end, naturally we need some particular assumptions on the noise coefficients. However, there is an overlap with the assumptions in \cite{mik,mik2}. But, there are cases where our assumptions hold, and those in \cite{mik,mik2} do not hold, and vice versa.

Our approach through a corresponding random partial differential equation has the advantage that we can do a "path by path" analysis and, thus, obtain a better understanding of the dependence of the solution on the Brownian path, since it is obtained by a fixed point argument, i.e., by iteration.

Now let us state the assumptions precisely. In the two-dimensional case, that is $d=2$, we take $A_i$ of the form

\begin{equation}\label{e2}A_iX:=\left(\begin{array}{c}a_{1i}\partial_1X_1+a_{2i}\partial_2X_1+a_{3i}\partial_1X_2+a_{4i}\partial_2X_2\\
a_{5i}\partial_1X_1+a_{6i}\partial_2X_1+a_{7i}\partial_1X_2+a_{8i}\partial_2X_2\end{array}\right),\ i=1,2,...,N.\end{equation} The coefficients  $a_{ji}, \ j=1,2,...,8,\ i=1,..,N,$  are given in the precise form below:
\begin{equation}\label{ho4}\begin{aligned}& a_{1i}=\sigma_i,\ a_{2i}=\sigma_i,\ a_{3i}=\mu_i\xi_1-\theta_i\xi_2,\ a_{4i}=\theta_i\xi_1+\mu_i\xi_2,\\&
 a_{5i}=-\mu_i\xi_1+\theta_i\xi_2,\ a_{6i}=-\theta_i\xi_1-\mu_i\xi_2,\ a_{7i}=\sigma_i,\ a_{8i}=\sigma_i.\end{aligned}\end{equation}Here, $(\xi_1,\xi_2)\in \mathbb{R}^2$ is the space variable.  $\sigma_{i}:\mathbb{R}_+\times \Omega \rightarrow\mathbb{R},\ \sigma_i=\sigma_i(t,\omega),\ i=1,2,...,N,$ are continuous functions, that are $\mathcal{F}_t-$addapted, with $\int_0^\infty \sigma^2_i(s)ds<\infty$, for each $\omega$. $\mu_i=\mu_i(t,\omega),\ i=1,...,N$ are random functions. Finally, $\theta_i,\ i=1,2,...,N$ are positive constants.

For the three-dimensional case, i.e., $d=3$, we take $A_i$ of the form

\begin{equation}\label{e700}A_iX:= \sigma_i\mathbf{1}_3\cdot \nabla X+\theta_iX,\ i=1,2,...,N,\end{equation}where $\mathbf{1}_3$ is the vector in $\mathbf{R}^3$ with all its elements equal to one, and $\theta_i$ and $\sigma_i$ are as above.
\begin{remark}\label{remarca}We notice that, in $2-D$, for $N=2$,  the special case: $$a_{1i}\equiv a_{2i}\equiv a_{7i}\equiv a_{8i}\equiv \sigma_i,\ i=1,2,$$
$$a_{4i}=a_{5i}=0,\ i=1,2,$$and
$$-\partial_1a_{6i}=-\partial_2 a_{3i}=\theta_i>0,$$ corresponds to the  model for turbulence considered in \cite{mik}. The same holds for the 3-D case we consider here.
\end{remark}

In this work, we let $L^p(\mathbb{R}^d),\ 0<p<\infty$ denote  the space  of power $p-$Lebesgue integrable functions with the norm $|\cdot|_p$; by $W^{l,p}(\mathbb{R}^d)$ the corresponding Sobolev space; $H^1(\mathbb{R}^d)=W^{1,2}(\mathbb{R}^d)$; and by $C_b([0,\infty);L^p(\mathbb{R}^d)) $ the space of all bounded and continuous  $L^p-$valued functions, defined on $[0,\infty)$, with the sup norm. Sometimes we will omit to express the dependence on the space $\mathbb{R}^d$, in the notations, if this will not create any confusion.  We also set $\partial_i=\frac{\partial}{\partial \xi_i},\ i=1,...,d$; and $\mathbf{1}_d$ the vector in $\mathbb{R}^d$ with all its elements equal to one. 

As indicated above, our aim here is to show that equation (\ref{ho10}) has a global strong solution, for a.e. fixed $\omega$ and small enough initial data,  in the mild sense. To this end, we shall further develop the ideas in \cite{b1}, that treat the stochastic $3-D$ Navier-Stokes equation with diffusion coefficient linear in the solution, while we analyze the case with diffusion coefficient linear in the gradient of the solution. Besides proving existence and uniqueness of solutions to \eqref{ho10}, globally in time, for $\mathbb{P}-$a.e. fixed $\omega\in\Omega$, provided that the initial condition is small enough, we also prove their continuity in time, in the following sense: we shall prove that the solution is weakly$-*$ continuous with respect to the time variable. For the two-dimensional case in the dual of $L^\frac{1}{\gamma}(\mathbb{R}^2)\cap L^\frac{2p}{3p-4}(\mathbb{R}^2)$, for some $\frac{4}{3}<p<2$ and $0<\gamma<1$; while for the three-dimensional case, we shall show that it is weakly$-*$ continuous in the dual of $\left(L^3(\mathbb{R}^3)\cap L^\frac{3p}{4p-6}(\mathbb{R}^3)\right)^3$ for some $\frac{3}{2}<p<3$, see Theorem \ref{t1} and \ref{t2} below, respectively. Then,  we deduce the existence of a   solution of the $2-D$ Navier-Stokes equation in vorticity form, which is adapted to the Brownian filtration,  up to some  stopping time; and a similar result concerning the $3-D$ Navier-Stokes equations, in Section \ref{33} below. We emphasize that, when studing turbulence,  the vorticity is a tool of central importance. Therefore, treating the Navier-Stokes equation in the vorticity form and obtaining existence and uniqueness results for the model is of high interest for understanding turbulence.

The structure of the paper is as follows: In Section 2 we derive the transformed equation, which is no longer a stochastic PDE, but a deterministic PDE with a random parameter. In Section 3, we concentrate on the 2D-case and in Section 4 on the 3D-case. In Section 5, we prove the existence of a solution to the original equation \eqref{e1}, which is adapted to the filtration, but exists only up to some stopping time. This is done both in the 2D- and 3D-case. 

\section{The transformed equation}

Consider the vorticity function (the vorticity field, for the three-dimensional case)
$$U:=\nabla\times X=curl\ X;$$ and apply the $curl$ operator to equation (\ref{e1}).

 In the two-dimensional case, taking advantage of  the form of $a_{1i},...,a_{8i},\ i=1,2,...,N$,  in \eqref{ho4}, we obtain 
\begin{equation}\label{e}\left\{\begin{array}{l}\displaystyle dU=\Delta Udt+(X\cdot \nabla)Udt+\sum_{i=1}^N(B_i(t)+\theta_iI)U d\beta_i \text { in }(0,\infty)\times\mathbb{R}^2,\\
U(0,\xi)=U_0(\xi)=(curl\ x)(\xi),\ \xi\in \mathbb{R}^2,\end{array}\right.\ \end{equation}where, for all $t\geq0,$  $B_i(t):H^1(\mathbb{R}^2)\rightarrow L^2(\mathbb{R}^2)$ is defined as
\begin{equation}\label{e9}B_i(t)f:=\sigma_{i}(t)\mathbf{1}_2\cdot \nabla f,\ i=1,2,...,N.\end{equation}

While, in the three-dimensional case, we get by the special form of $A_i$ in (\ref{e700}),
\begin{equation}\label{e701}\left\{\begin{array}{l}\displaystyle dU=\Delta Udt+(X\cdot \nabla)Udt-(U\cdot \nabla)Xdt+\sum_{i=1}^N(B_i(t)+\theta_iI)U d\beta_i \text { in }(0,\infty)\times\mathbb{R}^3,\\
U(0,\xi)=U_0(\xi)=(curl\ x)(\xi),\ \xi\in \mathbb{R}^3,\end{array}\right.\ \end{equation}where, for all $t\geq0,$  $B_i(t): (H^1(\mathbb{R}^3))^3\rightarrow (L^2(\mathbb{R}^3))^3$ is defined as
\begin{equation}\label{e703}B_i(t)U:=\sigma_{i}(t)\mathbf{1}_3\cdot \nabla U,\ i=1,2,...,N.\end{equation}

We begin with  some useful observations concerning the operators introduced above, and state them in the following lemmas. First of all, since the functions $\sigma_i$ do not depend on the space variable $\xi$, we immediately see that $B_i$ commutes with $B_j$, for all $i,j=1,2,...,N$, also it commutes with the laplacean $\Delta$.
\begin{lemma}\label{l12} Let $d=2,3$.  For all $i=1,2,...,N,$ the operators $B_i(t),\ t\geq0,$ are skew-adjoint; and they generate $C_0$-groups, denoted by $e^{sB_i(t)},\ s\in\mathbb{R},\ t\geq0,\ i=1,2,...,N.$ Moreover, for all $1<q<\infty$, we have
\begin{equation}\label{e40}|e^{sB_i(t)}f|_q=|f|_q,\ \forall f\in L^q(\mathbb{R}^d),\ s\in\mathbb{R},\ t\geq0, i=1,2,...,N;\end{equation}and
\begin{equation}\label{e41}\begin{aligned}|\nabla(e^{sB_i(t)}f)|_q= |\nabla f|_q,\ \forall f\in W^{1,q}(\mathbb{R}^d), \ s\in\mathbb{R},\ t\geq0, i=1,2,...,N.\end{aligned}\end{equation}
\end{lemma}

\begin{proof}We shall argue likewise  in \cite[(B2) and (B3)]{i1}, since the operators $B_i$ are of the same type as those in \cite{i1}. Thus,  one may show that the operators $B_i(t)$ satisfy $B_i^*(t)=-B_i(t),$ where $B_i^*(t)$ stands for the adjoint operator of $B_i(t)$ in $L^2(\mathbb{R}^d)$; and they generate  $C_0-$groups in $L^2(\mathbb{R}^2)$ for each $t\geq0,\ i=1,2,...,N.$ Besides this, as in \cite[(B3)]{i1}  one may compute that
\begin{equation}\label{e10}e^{sB_i(t)}f(\xi)=f(z_i(s,t,\xi)),\ s\in\mathbb{R},\  t\geq0,\ \xi\in\mathbb{R}^d,\end{equation}
where
\begin{equation}\label{e13}z_i(s,t,\xi)=\sigma_{i}(t)s\mathbf{1}_d+\xi,\end{equation}for $i=1,2,...,N$, $s\in\mathbb{R},\ t\geq0, \xi\in\mathbb{R}^d.$

We go on noticing that the Jacobian of the transformation $z_i,\ i=1,2,...,N,$ is equal to one. This implies that, for each $f\in L^q(\mathbb{R}^d),\ 0<q<\infty,$ we have
\begin{equation}|e^{sB_i}f|_q=\left(\int_{\mathbb{R}^d}|f(z_i(s,t,\xi)|^qd\xi\right)^\frac{1}{q}=|f|_q,\ i=1,2,...,N,\end{equation}so (\ref{e40}) is proved.
In order to conclude with the proof, we notice that $\nabla$ commutes with $B_i(t)$, and so, via the above equality, for each $f\in W^{1,q}(\mathbb{R}^d)$, one may easily deduce (\ref{e41}) as-well.
\end{proof}

We may get a similar result concerning the operators $e^{\int_0^t B_i(s)d\beta_i},\ i=1,2,...,N$ and $e^{\theta_i\int_0^t B_i(s)ds},\ i=1,2,...,N.$ More precisely,
\begin{corollary}\label{cor} Let $d=2,3$.  For all $i=1,2,...,N,$ one may well-define the exponential $e^{\int_0^t B_i(s)d\beta_i}$ and $e^{\theta_i\int_0^t B_i(s)ds}.$ Moreover, for all $1<q<\infty$, we have
\begin{equation}\label{edi40}|e^{\int_0^tB_i(s)d\beta_i}f|_q=|e^{\theta_i\int_0^tB_i(s)ds}f|_q=|f|_q,\ \forall f\in L^q(\mathbb{R}^d),\ s\in\mathbb{R},\ t\geq0, i=1,2,...,N;\end{equation}and
\begin{equation}\label{edi41}\begin{aligned}|\nabla e^{\int_0^tB_i(s)d\beta_i}f|_q=|\nabla e^{\theta_i\int_0^tB_i(s)ds}f|_q=|\nabla f|_q,\ \forall f\in W^{1,q}(\mathbb{R}^d), \ s\in\mathbb{R},\ t\geq0, i=1,2,...,N.\end{aligned}\end{equation}
\end{corollary}

\begin{proof}Note that 
$$e^{\int_0^t B_i(s)d\beta_i}f=e^{\beta_i(t)B_i(t)-\int_0^t\beta_i(s)B_i(s)ds}f,$$and that
$$\int_0^t\beta_i(s)B_i(s)fds=\int_0^t\beta_i(s)\sigma_i(s)ds \mathbf{1}_d\cdot \nabla f \text{ and } \theta_i\int_0^tB_i(s)fds=\int_0^t\theta_i\sigma_i(s)ds {1}_d\cdot \nabla f.$$ This means that, in fact, both $\int_0^t\beta_i(s)B_i(s)\ \cdot ds$ and $\theta_i\int_0^tB_i(s)\ \cdot ds$ are of similar form with $B_i$, with $\sigma_i$ replaced by $\int_0^t\beta_i(s)\sigma_i(s)ds$ and by $\int_0^t\theta_i\sigma_i(s)ds$, respectively. Therefore, arguing similarly as in the proof of Lemma \ref{l12}, one may show that indeed the exponential $e^{\int_0^t B_i(s)d\beta_i}$ and $e^{\theta_i\int_0^t B_i(s)ds}$ are well-defined and (\ref{edi40}) and (\ref{edi41}) hold true.

\end{proof}

Finally, let us show that we may well-define the exponential $e^{\pm\int_0^tB_i^2(s)ds},\ t\geq0, \  i=1,2,...,N,$ where $B_i$ are given by (\ref{e9}) or (\ref{e703}).
\begin{lemma}\label{l13}Let  $d=2,3$ and $1<q<\infty$. Then, it is possible to define the operator $e^{\frac{1}{2}\int_0^tB_i^2(s)ds},\ t\geq0,\  i=1,2,...,N,$ that is a contraction on $L^q(\mathbb{R}^d)$, i.e.
\begin{equation}\label{e401}\left|e^{\frac{1}{2}\int_0^tB_i^2(s)ds}f\right|_q\leq |f|_q,\ \forall f\in L^q(\mathbb{R}^d),\ t\geq0,\ i=1,2,...,N.\end{equation}Also, we have
\begin{equation}\label{e404}\left|\nabla\left(e^{\frac{1}{2}\int_0^tB_i^2(s)ds}f\right)\right|_q\leq |\nabla f|_q,\ \forall f\in W^{1,q}(\mathbb{R}^d),\ t\geq0,\ i=1,2,...,N.\end{equation}

Besides this, for each  $t\geq0$ and  $i=1,2,...,N$, the operator  $e^{\frac{1}{2}\int_0^tB_i^2(s)ds}$ is one-to-one on $L^q(\mathbb{R}^d)$. Therefore,  $e^{\frac{1}{2}\int_0^tB_i^2(s)ds}$ admits a left inverse, denoted by $e^{-\frac{1}{2}\int_0^tB_i^2(s)ds}$, for which there exists some positive constant $\mathcal{B}$ such that
\begin{equation}\label{e403}\left|e^{-\frac{1}{2}\int_0^tB_i^2(s)ds}f\right|_q\leq \mathcal{B} |f|_q,\ \forall f\in e^{\int_0^tB_i^2(s)ds}[L^q(\mathbb{R}^d)],\ t\geq0,\ i=1,2,...,N,\end{equation}and
\begin{equation}\label{e405}\left|\nabla\left(e^{-\frac{1}{2}\int_0^tB_i^2(s)ds}f\right)\right|_q\leq \mathcal{B} |\nabla f|_q,\ \forall f\in e^{\int_0^tB_i^2(s)ds}[W^{1,q}(\mathbb{R}^d)],\ t\geq0,\ i=1,2,...,N.\end{equation}
\end{lemma}
\begin{proof} Let any $1<q<\infty$ and $d=2,3$.  Set $O:H^1(\mathbb{R}^d)\rightarrow L^2(\mathbb{R}^d)$ as
$$O f:=\sum_{i=1}^d\partial_{i}f.$$ Notice that we have the equality
\begin{equation}\label{e400}\int_0^tB_i^2(s)ds=\int_0^t\sigma^2_i(s)ds\ O^2.\end{equation}
In the next lines we want to show that $O^2$ generates a $C_0-$semigroup in $L^q$. To this end we shall apply the Hille-Yosida's theorem. So, let any $f\in L^q$ and $\lambda>0$. We search for some $g$ such that
\begin{equation}\label{e305}\lambda g-O^2 g=f,\end{equation}or, equivalently,
\begin{equation}\label{e304}g-\frac{1}{\lambda}O^2g=\frac{1}{\lambda}f.\end{equation}To this purpose,  arguing as in \cite[(B3)]{i1}, we get that the solution $h$ to
\begin{equation}\label{e300}h+\frac{1}{\sqrt{\lambda}}Oh=\frac{1}{\lambda}f,\end{equation}is given by the formula
\begin{equation}\label{e301}h(\xi)=\frac{1}{\lambda}\int_0^\infty e^{-s}f\left(-\frac{1}{\sqrt{\lambda}}s\mathbf{1}_d+\xi\right)ds,\ \xi\in\mathbb{R}^d.\end{equation} Likewise, the solution $g$ to
\begin{equation}\label{e302}g-\frac{1}{\sqrt{\lambda}}Og=h,\end{equation}is given as
\begin{equation}\label{e303}g=\frac{1}{\lambda}\int_0^\infty e^{-\sigma}\left[\int_0^\infty e^{-s}f\left(\frac{1}{\sqrt{\lambda}}(\sigma-s)\mathbf{1}_d+\xi\right)ds\right]d\sigma\end{equation}Now, by (\ref{e300}) and (\ref{e302}) we see that $g$ satisfies (\ref{e304}), or, equivalently, (\ref{e305}). Thus,
\begin{equation}\label{e306}\begin{aligned}&|(\lambda I-O^2)^{-1}f|_q^q=|g|_q^q\\&
=\frac{1}{\lambda^q}\int_{\mathbb{R}^2}\left|\int_0^\infty\int_0^\infty e^{-s} e^{-\sigma}f\left(\frac{1}{\sqrt{\lambda}}(\sigma-s)\mathbf{1}_d+\xi\right)ds d\sigma\right|^qd\xi\\&
=\frac{1}{\lambda^q}\int_{\mathbb{R}^2}\left|\int_0^1\int_0^1 f\left(\frac{1}{\sqrt{\lambda}}(\ln s-\ln \sigma)\mathbf{1}_d+\xi\right)ds d\sigma\right|^qd\xi\\&
\text{(recalling that $q>1$, it follows by Jensen's inequality that)}\\&
\leq\frac{1}{\lambda^q}\int_{\mathbb{R}^2}\int_0^1\int_0^1 \left|f\left(\frac{1}{\sqrt{\lambda}}(\ln s-\ln \sigma)\mathbf{1}_d+\xi\right)\right|^qds d\sigma d\xi\\&
=\frac{1}{\lambda^q}\int_0^1\int_0^1\left[\int_{\mathbb{R}^2} \left|f\left(\frac{1}{\sqrt{\lambda}}(\ln s-\ln \sigma)\mathbf{1}_d+\xi\right)\right|^qd\xi\right]ds d\sigma \\&
\text{(noticing that the Jacobian of the transformation is equal to one)}\\&
=\frac{1}{\lambda^q}\int_0^1\int_0^1\left[\int_{\mathbb{R}^2} \left|f\left(\xi\right)\right|^qd\xi\right]ds d\sigma=\frac{1}{\lambda^q}|f|_q^q.
\end{aligned}\end{equation}Hence,
$$|(\lambda I-O^2)^{-1}f|_q\leq \frac{1}{\lambda}|f|_q,\ \forall \lambda>0,\ f\in L^q.$$Therefore, $O^2$ generates a $C_0-$ analytic semigroup on $L^q$, denoted by $e^{sO^2},\ s\geq0,$ of contractions, i.e., $|e^{sO^2}f|_q\leq|f|_q,\ f\in L^q,\ s\geq0.$ By \cite{anal}, $e^{sO^2},\ s\geq0,$ is analytic. Hence, recalling (\ref{e400}), we conclude that we may well-define the exponential $e^{\frac{1}{2}\int_0^tB_i^2(s)ds},\ i=1,2,...,N,$ and we have
\begin{equation}\left|e^{\frac{1}{2}\int_0^tB_i^2(s)ds}f\right|_q\leq |f|_q,\ \forall f\in L^q,\ t\geq0,\ i=1,2,...,N.\end{equation}Also, noticing that $\nabla$ commutes with $B_i(t)$, we have the first part of the lemma proved.

Regarding the last part of the lemma, we show that the semigroup generated by $O^2$ is one-to-one in $L^q$. To this aim, let us assume that for some $t>0$ and $f\in L^q(\mathbb{R}^d)$ we have
$$e^{tO^2}f=0.$$ Define 
$$t_0:=\inf\left\{s>0:\ e^{sO^2}f=0\right\}$$that is less or equal to $t$. Then by right strong continuity we have  
$$e^{t_0O^2}f=0.$$ We know that $e^{sO^2}f$ is, in fact, the solution $w(s)$ to the equation
$$\partial_t w=O^2w \text{ in }(0,\infty)\times \mathbb{R}^d;\ w(0)=f \text{ in }\mathbb{R}^d.$$ We have that $w(t_0)=0$, so, consequently, $\partial_tw(t_0)=0.$ But, then, since $\partial_t^2 w=O^2\partial_tw$, we get that $\partial_t^2w(t_0)=0$, as-well. Continuing with this argument, it yields that all the time-derivatives of $w$ in $t_0$ are equal to zero. Recalling the analyticity of the semigroup, we deduce that there exists $t'<t_0$ such that $e^{t'O^2}f=0$. Therefore, $t_0=0$ and so, $f=e^{0O^2}f=0.$ 

To conclude  the proof of the lemma, remember that

$$e^{\frac{1}{2}\int_0^t B_i^2(s)ds}=e^{\frac{1}{2}\int_0^t\sigma_i^2(s)ds\ O^2}.$$ Hence one can define the inverse $e^{-\frac{1}{2}\int_0^t B_i^2(s)ds}$, then, by  the inverse mapping theorem, the fact that $\int_0^\infty \sigma_i^2(s)ds$ is finite , and the fact that $\nabla$ commutes with $B_i(t),\ t\geq0,\ i=1,2,...,N,$ we immediately obtain (\ref{e403}) and (\ref{e405})

\end{proof}

\section{The existence results in the $2-D$ case}
Now, we  place ourselves  in the two dimensional case. Recall that $X$ may be expressed in terms of the vorticity $U$ as $X=K(U)$, where the operator $K$ is the Biot-Savart integral operator 
\begin{equation}\label{k}K(f)(\xi)=\frac{1}{2\pi}\int_{\mathbb{R}^2}\frac{(\xi-\bar{\xi})^\perp}{|\xi-\bar{\xi}|^2}f(\bar{\xi})d\bar{\xi},\ \xi\in\mathbb{R}^2.\end{equation} Hence,  equation (\ref{e}) may be equivalently written as
\begin{equation}\label{e2000}\left\{\begin{array}{l}\displaystyle dU=\Delta Udt+(K(U)\cdot \nabla)Udt+\sum_{i=1}^N(B_i(t)+\theta_i)U d\beta_i \text { in }(0,\infty)\times\mathbb{R}^2,\\
U(0,\xi)=U_0(\xi),\ \xi\in \mathbb{R}^2.\end{array}\right.\ \end{equation}

In order to reduce the SPDE (\ref{e2000}) to a random PDE, we consider the rescale
\begin{equation}\label{y*} U:=e^{\sum_{i=1}^N\left[\int_0^tB_i(s)d\beta_i+\theta_i\beta_i-\frac{1}{2}\int_0^tB_i^2(s)ds-\frac{1}{2}\theta_i^2t-\theta_i\int_0^tB_i(s)ds\right]}y.\end{equation} By Lemmas \ref{l12},\ \ref{l13} and Corollary \ref{cor}, we have that the operator 
\begin{equation}\label{e27*}\Gamma(t):=e^{\sum_{i=1}^N\left[\int_0^tB_i(s)d\beta_i+\theta_i\beta_i-\frac{1}{2}\int_0^tB_i^2(s)ds-\frac{1}{2}\theta_i^2t-\theta_i\int_0^tB_i(s)ds\right]}\end{equation} is well-defined on $L^q(\mathbb{R}^2)$, $1<q<\infty$, and it is left invertible. We set $\Gamma^{-1}$ for its inverse. 

Since for all $\phi\in L^p$ we have
$$e^{\sum_{i=1}^N[\int_0^t B_i(s)d\beta_i-\frac{1}{2}\int_0^t B_i^2(s)ds]}\phi=\phi+\int_0^t e^{\sum_{i=1}^N[\int_0^s B_i(\tau)d\beta_i-\frac{1}{2}\int_0^s B_i^2(\tau)d\tau]}\left(\sum_{i=1}^N B_i(t)\phi d\beta_i\right),$$ simple computations show that
\begin{equation}\label{edi1}d\Gamma(t) \phi=\Gamma(t)\left[\sum_{i=1}^N(B_i(t)+\theta_i)\phi d\beta_i(t)\right].\end{equation}

Then, similarly as in   \cite[Proposition 3.23 (iii)]{i1}, one may deduce that $y$ satisfies 
\begin{equation}\label{e26*}\begin{aligned}\frac{d y}{dt}=&\Delta y(t)+\Gamma^{-1}(t)[K(\Gamma(t)y(t))\cdot \nabla](\Gamma(t)y(t)),\ t>0;\ y(0)=U_0.\end{aligned}\end{equation}
We write equation (\ref{e26*}) in the mild formulation as
\begin{equation}\label{e52*}y(t)=G(y(t)):=e^{t\Delta}U_0+F(y)(t),\ t\geq0,\end{equation}where
\begin{equation}\label{e53*}\begin{aligned}F(f)(t):&=\int_0^te^{(t-s)\Delta}\Gamma^{-1}(s)[K(\Gamma(s)f(s))\cdot \nabla](\Gamma(s)f(s))ds,\ t\geq0.\end{aligned}\end{equation}Here,
$$(e^{t\Delta}g)(\xi):=\frac{1}{4\pi t}\int_{\mathbb{R}^2}e^{-\frac{|\xi-\overline{\xi}|^2}{4t}}g(\overline{\xi})d\overline{\xi},\ t\geq0,\ \xi\in\mathbb{R}^2.$$
For latter purpose, one can easily show that for $1<\alpha\leq \beta<\infty$, we have, for some $c>0$, the estimates

\begin{equation}\label{e55} |e^{t\Delta}g|_\beta\leq ct^{\frac{1}{\beta}-\frac{1}{\alpha}}|g|_\alpha,\ g\in L^\alpha(\mathbb{R}^2),\end{equation}and
\begin{equation}\label{e56}|\partial_je^{t\Delta}g|_\beta\leq c t^{\frac{1}{\beta}-\frac{1}{\alpha}-\frac{1}{2}}|g|_\alpha,\ u\in L^\alpha(\mathbb{R}^2),\ j=1,2.\end{equation}

The following theorem is the main result of this work concerning the 2-D case.
\begin{theorem}\label{t1}Let $\frac{4}{3}<p<2$ and $0<\frac{3}{2}-\frac{2}{p}<\gamma<1-\frac{1}{p}<\frac{3}{2}-\frac{1}{p}<1$. Let $\Omega_0:=\left\{\eta_\infty<\infty\right\}$ and consider  (\ref{e52*}) for fixed $\omega\in \Omega_0$. Then,  $\mathbb{P}(\Omega_0)=1$ and there is a positive constant $C$ independent of $\omega\in\Omega_0$ such that, if $U_0\in L^\frac{1}{1-\gamma}(\mathbb{R}^2)$ is such as
\begin{equation}\label{e1100}\eta_\infty |U_0|_\frac{1}{1-\gamma}\leq C,\end{equation}then the random equation (\ref{e52*}) has a unique solution $y\in \mathcal{Z}_p$ which satisfies
$$[K(\Gamma y)\cdot \nabla](\Gamma y)\in L^1(0,\infty; L^\frac{2p}{4-p}(\mathbb{R}^2)).$$ Here
$$\eta_\infty:=  e^{\sup_{0\leq s<\infty}\sum_{i=1}^N[\beta_i(s)\theta_i-\frac{s}{4}\theta_i^2]},$$ 
 and $\mathcal{Z}_p$ is defined by
\begin{equation}\label{Z}\mathcal{Z}_p:=\left\{f=f(t,\xi):\ t^{1-\frac{1}{p}-\gamma}f\in C_b([0,\infty);L^p(\mathbb{R}^2)),\ t^{\frac{3}{2}-\frac{1}{p}-\gamma}\partial_j f\in C_b([0,\infty);L^p(\mathbb{R}^2)),\ j=1,2\right\}.\end{equation}

Moreover, for each $\phi \in L^\frac{1}{\gamma}(\mathbb{R}^2)\cap L^\frac{2p}{3p-4}(\mathbb{R}^2)$, the function 
$$t\rightarrow \int_{\mathbb{R}^2}y(t,\xi)\phi(\xi)d\xi$$ is continuous on $[0,\infty)$. The map $U_0\rightarrow y$ is Lipschitz from $L^\frac{1}{1-\gamma}(\mathbb{R}^2)$ to $\mathcal{Z}_p$ .

In particular, the vorticity equation (\ref{e2000}) has a unique solution $U$ such that $\Gamma^{-1} U\in\mathcal{
Z}_p.$
\end{theorem}
We notice that, likewise in \cite[Remark 1.2]{b1} one can show that the condition (\ref{e1100}) is not void. The proof of Theorem \ref{t1} will be given in Section \ref{s3} below.

To prove our theorem, we shall rely on the following  two immediate results concerning the operators $K$ and $\Gamma$ introduced by (\ref{k}) and (\ref{e27*}), respectively.
\begin{lemma}For each $1<q<\infty$  we have
\begin{equation}\label{e30}|\Gamma(t)f|_q\leq \mathcal{B}  e^{\sum_{i=1}^N\left[\beta_i(t)\theta_i-\frac{t}{2}\theta_i^2\right]}\left|f\right|_q \text{ and }|\Gamma^{-1}(t)f|_q\leq  e^{\sum_{i=1}^N\left[-\beta_i(t)\theta_i+\frac{t}{2}\theta_i^2\right]}\left|f\right|_q\end{equation} for all $ t\geq0, f\in L^q(\mathbb{R}^2)$; and
\begin{equation}\label{e31}|\nabla(\Gamma(t)f)|_q\leq \mathcal{B}  e^{\sum_{i=1}^N\left[\beta_i(t)\theta_i-\frac{t}{2}\theta_i^2\right]}\left|\nabla f\right|_q, \text{ and } |\nabla(\Gamma^{-1}(t)f)|_q\leq e^{\sum_{i=1}^N\left[-\beta_i(t)\theta_i+\frac{t}{2}\theta_i^2\right]}\left|\nabla f\right|_q\end{equation}
for all $t\geq0,\  f\in W^{1,q}(\mathbb{R}^2).$
\end{lemma}
\begin{proof}Recalling the definition of $\Gamma(t)$ given by (\ref{e27*}), it is easy to see that the inequality in (\ref{e30}) yields from (\ref{e40}), (\ref{edi40}) (\ref{e403}) and (\ref{e401}), while (\ref{e31}) follows from (\ref{e41}),  (\ref{edi41}), (\ref{e405}) and (\ref{e404}). 
\end{proof}

\begin{lemma}\label{l2}Let $r=\frac{2p}{2-p},\ \frac{4}{3}<p<2,\ q=\frac{2r}{4+r}>1.$ Then we have
\begin{equation}\label{e50}\left|\Gamma^{-1}(t)[K(\Gamma(t)f)\cdot \nabla](\Gamma(t)f)\right|_q\leq   \mathcal{B}^2e^{\sum_{i=1}^N\left[\beta_i(t)\theta_i-\frac{t}{2}\theta_i^2\right]}|f|_{p}|\nabla f|_{p},\ \forall f\in W^{1,p}(\mathbb{R}^2).\end{equation}
\end{lemma}
\begin{proof}Notice that  we have $\frac{1}{q}=\frac{1}{r}+\frac{1}{p}$. By (\ref{e30}) and H$\ddot{o}$lder's inequality, we obtain
\begin{equation}\label{e110}\begin{aligned}&\left|\Gamma^{-1}(t)[K(\Gamma(t)f)\cdot \nabla](\Gamma(t)f)\right|_q\leq e^{\sum_{i=1}^N\left[-\beta_i(t)\theta_i+\frac{t}{2}\theta_i^2\right]}|K(\Gamma(t)f)|_{r}|\nabla(\Gamma(t)f)|_{p}\\&
\text{( recalling the classical estimate for the Riesz potentials (see \cite[p. 119]{rt}) )}\\&
\leq e^{\sum_{i=1}^N\left[-\beta_i(t)\theta_i+\frac{t}{2}\theta_i^2\right]}|\Gamma(t)f|_{\frac{2r}{2+r}}|\nabla(\Gamma(t)f)|_{p}\\&
\left(\text{ noticing that $\frac{2r}{2+r}=p$ and taking advantage of relations (\ref{e30}) and (\ref{e31})}\right)\\&
\leq  \mathcal{B}^2e^{\sum_{i=1}^N\left[\beta_i(t)\theta_i-\frac{t}{2}\theta_i^2\right]}|f|_p|\nabla f|_{p},
\end{aligned}\end{equation} thereby completing the proof.
\end{proof}

\subsection{Proof of Theorem \ref{t1}}\label{s3}

From now on, we fix $p, q, r$ as in Lemma   \ref{l2}, i.e.,
\begin{equation}\label{e51}\frac{4}{3}<p<2,\  r=\frac{2p}{2-p},\ q=\frac{2p}{4-p}>1.\end{equation}
In the following we shall estimate the quantities $|F(f(t))|_p$ and $|\nabla F(f(t))|_p,$ where $F$ is defined by (\ref{e53*}). To this end, since $p>q>1$, we may take in (\ref{e55}) $\alpha=q$ and $\beta=p$, to obtain that
\begin{equation}\label{e57}\begin{aligned}&|F(f)(t)|_p\leq\\&
\left|\int_0^te^{(t-s)\Delta}\Gamma^{-1}(s)[K(\Gamma(s)f(s))\cdot \nabla](\Gamma(s)f(s))ds\right|_p \\&
\leq c\int_0^t(t-s)^{\frac{1}{p}-\frac{1}{q}}|\Gamma^{-1}(s)[K(\Gamma(s)f(s))\cdot \nabla](\Gamma(s)f(s))|_qds\\&
(\text{ invoke (\ref{e50})})\\&
\leq c\mathcal{B}^2\int_0^t(t-s)^{\frac{1}{2}-\frac{1}{p}}e^{\sum_{i=1}^N[\beta_i(s)\theta_i-\frac{s}{2}\theta_i^2]}|f(s)|_p|\nabla f(s)|_pds\\&
=c\mathcal{B}^2\int_0^t(t-s)^{\frac{1}{2}-\frac{1}{p}}e^{\sum_{i=1}^N[\beta_i(s)\theta_i-\frac{s}{4}\theta_i^2]}e^{-\frac{s}{4}\theta_i^2}|f(s)|_p|\nabla f(s)|_pds\\&
\leq c\mathcal{B}^2\eta_t\max\left\{1,\frac{4}{\theta_i^2},\ i=1,2,...,N\right\}\int_0^t(t-s)^{\frac{1}{2}-\frac{1}{p}}s^{-\gamma}|f(s)|_p|\nabla f(s)|_pds,\end{aligned}\end{equation}where
\begin{equation}\label{e80}\eta_t:=e^{\sup_{0\leq s\leq t}\sum_{i=1}^N[\beta_i(s)\theta_i-\frac{s}{4}\theta_i^2]},\end{equation}and $\gamma>0$ was chosen such that
\begin{equation}\label{e81}1>\gamma>\frac{3}{2}-\frac{2}{p}>0.\end{equation}

In the same manner one may obtain also that
\begin{equation}\label{e60}\begin{aligned}|\nabla F(f)(t)|_p\leq& c\mathcal{B}^2\eta_t\max\left\{1,\frac{4}{\theta_i^2},\ i=1,2,...,N\right\}\int_0^t(t-s)^{-\frac{1}{p}}s^{-\gamma}|f(s)|_p|\nabla f(s)|_pds.\end{aligned}\end{equation}

Next, we consider the Banach space $\mathcal{Z}_p$ defined by
$$\mathcal{Z}_p:=\left\{f:\ t^{1-\frac{1}{p}-\gamma}f\in C_b([0,\infty);L^p),\ t^{\frac{3}{2}-\frac{1}{p}-\gamma}\partial_j f\in C_b([0,\infty);L^p),\ j=1,2\right\},$$ endowed with the norm
$$\|f\|:=\sup_{t>0}\left\{t^{1-\frac{1}{p}-\gamma}|f(t)|_p+t^{\frac{3}{2}-\frac{1}{p}-\gamma}|\nabla f(t)|_p\right\}.$$Easily seen, we have that
\begin{equation}\label{e61}|f(t)|_p|\nabla f(t)|_p\leq t^{\frac{2}{p}-\frac{5}{2}+2\gamma}\|f\|^2,\ \forall f\in \mathcal{Z}_p,\ t>0.\end{equation}

It yields from (\ref{e57}) that for $f\in \mathcal{Z}_p$, we have
\begin{equation}\label{e62}\begin{aligned}|F(f)(t)|_p&\leq  c\mathcal{B}^2\eta_t \max\left\{1,\frac{4}{\theta_i^2},\ i=1,2,...,N\right\} \|f\|^2 \int_0^t (t-s)^{\frac{1}{2}-\frac{1}{p}}s^{\frac{2}{p}-\frac{5}{2}+\gamma}ds\\&
=c\mathcal{B}^2\eta_t \max\left\{1,\frac{4}{\theta_i^2},\ i=1,2,...,N\right\} \|f\|^2 t^{\frac{1}{p}-1+\gamma}\int_0^1 (1-s)^{\frac{1}{2}-\frac{1}{p}}s^{\frac{2}{p}-\frac{5}{2}+\gamma}ds\\&
=t^{\frac{1}{p}-1+\gamma}c\mathcal{B}^2\eta_t \max\left\{1,\frac{4}{\theta_i^2},\ i=1,2,...,N\right\}  B\left(\frac{2}{p}-\frac{3}{2}+\gamma,\frac{3}{2}-\frac{1}{p}\right)\|f\|^2,\end{aligned}\end{equation}where $B(x,y)$ is the classical beta function. Note that $B\left(\frac{2}{p}-\frac{3}{2}+\gamma,\frac{3}{2}-\frac{1}{p}\right)$ is finite by virtue of (\ref{e51}) and (\ref{e81}). 

Similarly, by (\ref{e60}), we have 
\begin{equation}\label{e97}\begin{aligned}|\nabla F(f)(t)|_p&\leq t^{\frac{1}{p}-\frac{3}{2}+\gamma}c\mathcal{B}^2\eta_t \max\left\{1,\frac{4}{\theta_i^2},\ i=1,2,...,N\right\}  B\left(\frac{2}{p}-\frac{3}{2}+\gamma,1-\frac{1}{p}\right)\|f\|^2.\end{aligned}\end{equation}

Hence, (\ref{e62}) and (\ref{e97}) give
\begin{equation}\label{e98}\|F(f)\|\leq \mathcal{C}\eta_\infty\|f\|^2,\end{equation}where $\eta_\infty:=\sup_{t\geq0}\eta_t,$ and
$$\mathcal{C}:=c\mathcal{B}^2\max\left\{1,\frac{4}{\theta_i^2},\ i=1,2,...,N\right\}\max\left\{B\left(\frac{2}{p}-\frac{3}{2}+\gamma,\frac{3}{2}-\frac{1}{p}\right),\ B\left(\frac{2}{p}-\frac{3}{2}+\gamma,1-\frac{1}{p}\right)\right\}.$$

By (\ref{e55})-(\ref{e56}), we have
$$|e^{t\Delta}U_0|_p\leq c t^{\frac{1}{p}-1+\gamma}|U_0|_\frac{1}{1-\gamma},\ t>0$$and
$$|\nabla e^{t\Delta}U_0|_p\leq ct^{\frac{1}{p}-\frac{3}{2}+\gamma}|U_0|_\frac{1}{1-\gamma},\ t>0,$$where we recall that by (\ref{e81}) $\gamma$ was chosen such that $0<\gamma<1$. Therefore,
\begin{equation}\label{e100}\|e^{t\Delta}U_0\|\leq c|U_0|_\frac{1}{1-\gamma}.\end{equation}

We deduce that, (\ref{e98}) together with (\ref{e100}) imply that
\begin{equation}\label{e101}\|G(f)\|\leq c|U_0|_\frac{1}{1-\gamma}+\mathcal{C}\eta_\infty\|f\|^2,\ \forall f\in \mathcal{Z}_p.\end{equation}

Let    $0<\rho =\rho(\omega)$ such that
\begin{equation}\label{ro}\rho<\frac{1}{4c\mathcal{C}}\frac{1}{\eta_\infty}.\end{equation}Then set
$$\Sigma:=\left\{f\in \mathcal{Z}_p:\ \|f\|\leq R^*\right\},$$where
\begin{equation}R^*(=R^*(\omega))= 2c\rho.\end{equation}
Assuming that 
\begin{equation}\label{e103}|U_0|_\frac{1}{1-\gamma}\leq \rho,\end{equation}we  see, via (\ref{e101}), that $G(\Sigma)\subset \Sigma$.

Now, let $f,\bar{f}\in \Sigma.$ We want to estimate the difference $\|G(f)-G(\bar{f})\|$. We have,
as in \cite[Eqs. (2.32)-(2.33)]{b1} and (\ref{e57}), that
\begin{equation}\label{e104}\begin{aligned}&\|G(f)-G(\bar{f})\|\\&
=\left\|\int_0^\cdot e^{(\cdot-s)\Delta}\left(\Gamma^{-1}(s)[K(\Gamma(s)f(s)\cdot\nabla](\Gamma(s)f(s))-\Gamma^{-1}(s)[K(\Gamma(s)\bar{f}(s)\cdot\nabla](\Gamma(s)\bar{f}(s))\right)ds\right\|\\&
\leq \mathcal{C}_1\eta_\infty R^*\|f-\bar{f}\|,\end{aligned}\end{equation}for some positive $\mathcal{C}_1>0$. In addition, if $\rho$ satisfies
$$\rho<\frac{1}{2c\mathcal{C}_1}\frac{1}{\eta_\infty},$$then we see by (\ref{e104}) that $G$ is a contraction on $\Sigma$.

Hence, if $\rho>0$ is such that
\begin{equation}\label{rho}\rho<\min\left\{\frac{1}{4c\mathcal{C}}\frac{1}{\eta_\infty}; \ \frac{1}{2c\mathcal{C}_1}\frac{1}{\eta_\infty}\right\}\end{equation}and $|U_0|_\frac{1}{1-\gamma}<\rho$, then there is a unique solution $y\in \Sigma$ to (\ref{e52*}). So, $C$ in (\ref{e1100}) is a nonrandom constant such as 
$$C<\min\left\{\frac{1}{4c\mathcal{C}}; \ \frac{1}{2c\mathcal{C}_1}\right\}.$$

The proof of the last part of the Theorem \ref{t1} goes similarly as  the corresponding part of the proof of \cite[Theorem 1.1]{b1}. That is why, we only sketch it.  

By \eqref{e52*}, for all $\phi\in C_0^\infty(\mathbb{R}^2)$, we have that
$$\begin{aligned}\int_{\mathbb{R}^2}y(t,\xi)\phi(\xi)d\xi=&\int_{\mathbb{R}^2}e^{t\Delta}U_0(\xi)\phi(\xi)d\xi\\&
+\int_0^t\int_{\mathbb{R}^2}\Gamma^{-1}(s)[K(\Gamma(s)y(s))\cdot \nabla](\Gamma(s)y(s))e^{(t-s)\Delta}\phi(\xi)d\xi ds.\end{aligned}$$
Since $|e^{t\Delta}\phi|_{\tilde{p}}\leq |\phi|_{\tilde{p}},$ for all $\phi\in L^{\tilde{p}}(\mathbb{R}^2),\ 1\leq \tilde{p}<\infty,\ t\geq0$, it follows by relations  \eqref{e103}, and \eqref{e50}  \eqref{e61} and \eqref{e80},  that 
\begin{equation}\label{ho1}|\int_{\mathbb{R}^2}e^{t\Delta} U_0(\xi)\phi(\xi)d\xi|\leq \rho\ |\phi|_\frac{1}{\gamma},\end{equation}
and
\begin{equation}\label{ho2}|\int_0^t\int_{\mathbb{R}^2}\Gamma^{-1}(s)[K(\Gamma(s)y(s))\cdot \nabla](\Gamma(s)y(s))e^{(t-s)\Delta}\phi(\xi)d\xi ds| \leq \mathcal{B}^2\eta_\infty t^{\frac{2}{p}-\frac{3}{2}+2\gamma}\|y\|^2|\phi|_\frac{2p}{3p-4},\ \forall t >0.\end{equation}
Hence, via \eqref{ho1} and \eqref{ho2}, we arrive at
$$|\int_{\mathbb{R}^2}y(t,\xi)\phi(\xi)d\xi|\leq CT^{\frac{2}{p}-\frac{3}{2}+2\gamma}(|\phi|_\frac{1}{\gamma}+|\phi|_\frac{2p}{3p-4}),\ \forall \phi\in L^\frac{1}{\gamma}\cap L^\frac{2p}{3p-4},\ t\in [0,T].$$ Besides this, since $t\rightarrow e^{t\Delta}U_0$ is continuous on $L^\frac{1}{1-\gamma}$, taking into account the above, we deduce that $t\rightarrow y(t)$ is $L^\frac{1}{\gamma} \cap L^\frac{2p}{3p-4}$ weakly continuous on $[0,\infty)$.

Now, let $U_0,\ \overline{U_0}$ satisfying \eqref{e103}. Denote by $y(t, U_0),\ y(t,\overline{U_0})\in \mathcal{Z}_p$ the corresponding solutions of \eqref{e52*} with initial data $U_0$ and $\overline{U_0}$, respectively. With similar arguments as in \eqref{e57}-\eqref{e98} and \eqref{e104}, we may show that
$$\|y(\cdot,U_0)-y(\cdot,\overline{U_0})\|\leq C|U_0-\overline{U_0}|_\frac{1}{1-\gamma}+\eta_\infty C_1 R^*\|y(\cdot,U_0)-y(\cdot,\overline{U_0})\|.$$Since $R^*C_1\eta_\infty$ was chosen to be strictly less than 1, we conclude that
$$\|y(\cdot,U_0)-y(\cdot,\overline{U_0})\|\leq \frac{C}{1-R^*C_1\eta_\infty}|U_0-\overline{U_0}|_\frac{1}{1-\gamma}.$$With other words, the map $U_0\rightarrow y(\cdot, U_0)$ is Lipschitz from $L^\frac{1}{1-\gamma}$ to $\mathcal{Z}_p$. \hfill$\Box$

\subsection{Global in time behavior of the solution}

Let us recall that, in virtue of \eqref{e52*}, we have that the solution $y$ to \eqref{e26*} satisfies
$$y(t)=e^{t\Delta}U_0+\int_0^te^{(t-s)\Delta}M(y(s))ds,$$where $M(y(s)):=\Gamma^{-1}(s)(K(\Gamma(s)y(s))\cdot \nabla)\Gamma(s)y(s)).$
It follows that
\begin{equation}\label{ho30}|y(t)|_\frac{1}{1-\gamma}\leq C|e^{t\Delta}U_0|_\frac{1}{1-\gamma}+\int_0^t|e^{(t-s)\Delta}M(y(s))|_\frac{1}{1-\gamma}ds,\end{equation}where we use \eqref{e55}, to obtain
\begin{equation}\label{ho31}|y(t)|_\frac{1}{1-\gamma}\leq C\left[|U_0|_\frac{1}{1-\gamma}+\int_0^t(t-s)^{1-\gamma-\frac{1}{\alpha}}|M(y(s))|_\alpha ds\right].\end{equation}

We take $\alpha=\frac{2p}{4-p}$,  and use relation \eqref{e50} and similar ideas as in \eqref{e57}, to deduce that
$$\int_0^t(t-s)^{1-\gamma-\frac{1}{\alpha}}|M(y(s))|_\alpha ds\leq C\eta_\infty\int_0^t (t-s)^{1-\gamma-\frac{4-p}{2p}}s^{-\gamma}|y(s)|_p|\nabla y(s)|_pds,$$where involving \eqref{e61}, it yields
\begin{equation}\label{ho33}\begin{aligned}\int_0^t(t-s)^{1-\gamma-\frac{1}{\alpha}}|M(y(s))|_\alpha ds &\leq  C\eta_\infty\int_0^t (t-s)^{1-\gamma-\frac{4-p}{2p}}\cdot s^{\frac{2}{p}-\frac{5}{2}+2\gamma-\gamma}ds \ \|y\|^2\\&
=C\eta_\infty B\left(\frac{2}{p}-\frac{5}{2}+\gamma+1,1-\gamma-\frac{4-p}{2p}+1\right)\|y\|^2,\ \forall t\geq0. \end{aligned}\end{equation}
By the choice  of $p$ and $\gamma$ in Theorem \ref{t1}, we see that the beta function $B\left(\frac{2}{p}-\frac{5}{2}+\gamma+1,1-\gamma-\frac{4-p}{2p}+1\right)$ is finite.

Hence, \eqref{ho31} and \eqref{ho33} imply that
\begin{equation}\label{ho34} |y(t)|_\frac{1}{1-\gamma}\leq C\left[\rho+\rho^2\right],\ \forall t\geq0,\end{equation}where $\rho$ is introduced by \eqref{rho}.

\subsection{A random version of the $2-D$ Navier-Stokes equation and existence of its solution, for small initial data}
 We keep on following the ideas in \cite[Section 3]{b1}. We fix in (\ref{e1}) the initial random variable $x$ by the formula
$$x=K(U_0),$$ where $U_0$ satisfies condition (\ref{e110}) for all $\omega\in \Omega_0$. Then we define the process $X$ by the formula
$$X(t)=K(U(t))=K(\Gamma(t)y(t)),\ t\geq0,$$where $y$ is the solution to (\ref{e52*}) who's existence and uniqueness is guaranteed by Theorem \ref{t1}. Since $U\in \mathcal{Z}_p$ (defined in Theorem \ref{t1}), recalling the arguments in (\ref{e110}) (i.e., via the Riesz potentials estimates we showed that $|K(f)|_r\leq |f|_\frac{2r}{2+r}$), we deduce that
$$|X(t)|_\frac{2p}{2-p}\leq C |U(t)|_p,\ t\geq0,$$and so
\begin{equation}\label{e111}t^{1-\frac{1}{p}-\gamma}X\in C_b([0,\infty);L^\frac{2p}{2-p}).\end{equation}

Furthermore, by the Carlderon-Zygmund inequality (see \cite[Theorem 1]{z}), we know that
$$|\nabla K(f)|_p\leq C| f|_p,\ \forall f\in L^p.$$ Thus taking once $f=U(t)$ then $f=\partial_j U(t)$ in the above inequality, and making use of the fact that $\Gamma^{-1} U\in \mathcal{Z}_p$, we get that
\begin{equation}\label{e112}t^{1-\frac{1}{p}-\gamma}\partial_i X\in C_b([0,\infty);L^p),\end{equation}and
\begin{equation}\label{e113}t^{\frac{3}{2}-\frac{1}{p}-\gamma}\partial_i\partial_j X\in C_b([0,\infty);L^p),\end{equation}for $i,j=1,2.$

By (\ref{e1100}) and (\ref{e103}) and the Fernique theorem, we see that both $|U_0|_\frac{1}{1-\gamma}$ and $R^*$ belong to $\cap_{r\geq 1}L^r(\Omega).$ Thus, since $y\in \Sigma$, we get via (\ref{e111})-(\ref{e113}) that
$$\begin{aligned}& t^{1-\frac{1}{p}-\gamma}X\in C_b([0,\infty);\ L^r(\Omega;L^\frac{2p}{2-p})),\\&
t^{1-\frac{1}{p}-\gamma}\partial_iX\in C_b([0,\infty);\ L^r(\Omega;L^p)),\\&
t^{\frac{3}{2}-\frac{1}{p}-\gamma}\partial_i\partial_jX\in C_b([0,\infty);\ L^r(\Omega;L^p)),\ \forall r\geq1, \ i,j=1,2.\end{aligned}$$

Finally, if in equation (\ref{e52*}) one applies the operator $K(\Gamma\cdot)$, we get for $X$ the equation
$$\begin{aligned}X(t)&=K(e^{t\Delta} \Gamma(t)\ curlx)+\int_0^t K\left(e^{(t-s)\Delta}\Gamma(t)\Gamma^{-1}(s)[K(\ curl X(s))\cdot \nabla](\ curl X(s))\right)ds,\ t\geq0.\end{aligned}$$
The above equation may be viewed as the random version of the Navier-Stokes equation (\ref{e1}). However, since $U_0$ is not $\mathcal{F}_0-$measurable, the process $t\rightarrow U(t)$ is not $\mathcal{F}_t-$adapted, and so $X$ is not $\mathcal{F}_t-$adapted, too. By Theorem \ref{t1} we know that the above equation has a unique solution.

\section{The existence results for the $3-D$ case}
Now, we place ourselves in the whole $\mathbb{R}^3$ space. In this case, the Biot-Savart integral operator is given as
$$K(u)(\xi):=-\frac{1}{4\pi}\int_{\mathbb{R}^3}\frac{\xi-\overline{\xi}}{|\xi-\overline{\xi}|^3}\times u(\overline{\xi})d\overline{\xi},\ \xi\in\mathbb{R}^3. $$ Hence, by (\ref{e701}), we get the following equation for the vorticity field $U$:
\begin{equation}\label{j1}\left\{\begin{array}{l}\displaystyle dU=\Delta Udt+[K(U)\cdot \nabla]Udt-(U\cdot \nabla)K(U)dt+\sum_{i=1}^N(B_i(t)+\theta_iI)U d\beta_i \text { in }(0,\infty)\times\mathbb{R}^3,\\
U(0,\xi)=U_0(\xi)=(curl\ x)(\xi),\ \xi\in \mathbb{R}^3.\end{array}\right.\ \end{equation}
Consider again the transformation 
\begin{equation}\label{j4} U(t):=e^{\sum_{i=1}^N\left[\int_0^tB_i(s)d\beta_i+\beta_i\theta_i-\frac{1}{2}\int_0^tB_i^2(s)ds-\frac{1}{2}\theta_i^2t-\theta_i\int_0^tB_i(s)ds\right]}y(t).\end{equation} By Lemmas \ref{l12}, \ref{l13} and Corollary \ref{cor}, we have that the operator 
\begin{equation}\label{j5}\Gamma(t):=e^{\sum_{i=1}^N\left[\int_0^tB_i(s)d\beta_i+\beta_i\theta_i-\frac{1}{2}\int_0^tB_i^2(s)ds-\frac{1}{2}\theta_i^2t-\theta_i\int_0^tB_i(s)ds\right]}\end{equation} is well-defined on $L^q(\mathbb{R}^3)$, $1<q<\infty$, and it is invertible. Again, set $\Gamma^{-1}$ for its inverse.

Then, likewise in \cite{l2}, one may show that $y$ satisfies 
\begin{equation}\label{j6}\begin{aligned}\frac{d y}{dt}=&\Delta y(t)+\Gamma^{-1}(t)[K(\Gamma(t)y(t))\cdot \nabla](\Gamma(t)y(t))-\Gamma^{-1}(t)(\Gamma y(t)\cdot\nabla)(K(\Gamma y(t));\ y(0)=U_0.\end{aligned}\end{equation}
We write equation (\ref{j6}) in the mild formulation as
\begin{equation}\label{j7}y(t)=G(y(t)):=e^{t\Delta}U_0+F(y)(t),\ t\geq0,\end{equation}where
\begin{equation}\label{j8}\begin{aligned}F(f)(t):=\int_0^te^{(t-s)\Delta}\Gamma^{-1}(s)[K(\Gamma(s)f(s))\cdot \nabla](\Gamma(s)f(s))ds-\int_0^t e^{(t-s)\Delta}\Gamma^{-1}(s)(\Gamma y(s)\cdot\nabla)(K(\Gamma y(s))ds,\ t\geq0.\end{aligned}\end{equation}Here,
$$(e^{t\Delta}g)(\xi):=\frac{1}{(4\pi t)^\frac{3}{2}}\int_{\mathbb{R}^3}e^{-\frac{|\xi-\overline{\xi}|^2}{4t}}g(\overline{\xi})d\overline{\xi},\ t\geq0,\ \xi\in\mathbb{R}^3.$$
One can easily show that for $1<\alpha\leq \beta<\infty$, we have, for some $c>0$, the estimates

\begin{equation}\label{j9} |e^{t\Delta}g|_\beta\leq ct^{\frac{3}{2}\left(\frac{1}{\beta}-\frac{1}{\alpha}\right)}|g|_\alpha,\ g\in L^\alpha(\mathbb{R}^3),\end{equation}and
\begin{equation}\label{j10}|\partial_je^{t\Delta}g|_\beta\leq c t^{\frac{3}{2}\left(\frac{1}{\beta}-\frac{1}{\alpha}\right)-\frac{1}{2}}|g|_\alpha,\ u\in L^\alpha(\mathbb{R}^3),\ j=1,2,3.\end{equation}

The following theorem is the counterpart, for the $3-D$ case, of the Theorem \ref{t1}.
\begin{theorem}\label{t2}Let $\frac{3}{2}<p<2$, and $\Omega_0:=\left\{\eta_\infty<\infty\right\}$ and consider  (\ref{j7}) for fixed $\omega\in \Omega_0$. Then,   $\mathbb{P}(\Omega_0)=1$ and there is a positive constant $C$ independent of $\omega\in\Omega_0$ such that, if $U_0\in L^\frac{3}{2}(\mathbb{R}^3)$ is such as
\begin{equation}\label{j13}\eta_\infty |U_0|_\frac{3}{2}\leq C,\end{equation}then the random equation (\ref{j7}) has a unique solution $y\in \mathcal{Z}_p$ which satisfies
$$[K(\Gamma y)\cdot \nabla](\Gamma y)-(\Gamma y\cdot\nabla)(K(\Gamma y))\in L^1(0,\infty; L^\frac{3p}{6-p}(\mathbb{R}^3)).$$ Here
$$\eta_\infty:=  e^{\sup_{0\leq s<\infty}\sum_{i=1}^N[\beta_i(s)\theta_i-\frac{s}{4}\theta_i^2]}$$ 
 and $\mathcal{Z}_p$ is defined by
\begin{equation}\label{j15}\mathcal{Z}_p:=\left\{f:\ t^{1-\frac{3}{2p}}f\in C_b([0,\infty);L^p(\mathbb{R}^3)),\ t^{\frac{3}{2}\left(1-\frac{1}{p}\right)}\partial_j f\in C_b([0,\infty);L^p(\mathbb{R}^3)),\ j=1,2,3\right\}.\end{equation}

Moreover, for each $\phi \in L^3(\mathbb{R}^3)\cap L^\frac{3p}{4p-6}(\mathbb{R}^3)$, the function 
$$t\rightarrow \int_{\mathbb{R}^3}y(t,\xi)\phi(\xi)d\xi$$ is continuous on $[0,\infty)$. The map $U_0\rightarrow y$ is Lipschitz from $L^\frac{3}{2}(\mathbb{R}^3)$ to $\mathcal{Z}_p$ .

In particular, the vorticity equation (\ref{j1}) has a unique solution $U$ such that $\Gamma^{-1} U\in\mathcal{
Z}_p.$
\end{theorem}
We notice that, likewise in \cite[Remark 1.2]{b1} one can show that the condition (\ref{j13}) is not void. 
\begin{proof}The proof follows by identical arguments as those in the proof of Theorem \ref{t1}, but, this time, with the Sobolev embeddings, Riesz potential estimates, and Calderon-Zygmund inequality corresponding to the $3-D$ case, that can be found in the proof of Theorem 1 in \cite{b1}. Therefore, no further details are given.
\end{proof}
\subsection{A random version of the $3-D$ Navier-Stokes equation}\label{33}
We go one  obtaining the counterparts of the results in the $2-D$ case, for the $3-D$ case as-well.
Concerning the random version of the $3-D$ Navier-Stokes equation, we fix in (\ref{e1}) the initial random variable $x$ by the formula
$$x=K(U_0),$$ where $U_0$ satisfies condition (\ref{j13}) for all $\Omega_0$. Then we define the process $X$ by the formula
$$X(t)=K(U(t))=K(\Gamma(t)y(t)),\ t\geq0,$$where $y$ is the solution to (\ref{j7}) who's existence and uniqueness is guaranteed by Theorem \ref{t2}. Since $U\in \mathcal{Z}_p$ (defined in Theorem \ref{t2}),  via the Riesz potentials estimates we may show that $|X(t)|_\frac{3p}{3-p}\leq |U(t)|_p,\ t\geq0,$ (see \cite[Eq. (3.3)]{b1}), and so
\begin{equation}\label{j30}t^{1-\frac{3}{2p}}X\in C_b([0,\infty);L^\frac{3p}{3-p}).\end{equation}

Furthermore, by the Carlderon-Zygmund inequality (see \cite[Theorem 1]{z}), we know that
$$|\nabla K(f)|_p\leq C| f|_p,\ \forall f\in L^p.$$ Thus taking once $f=U(t)$ then $f=\partial_j U(t)$ in the above inequality, and making use of the fact that $\Gamma^{-1} U\in \mathcal{Z}_p$, we get that
\begin{equation}\label{j20}t^{\frac{3}{2}\left(1-\frac{1}{p}\right)}\partial_i X\in C_b([0,\infty);L^p),\end{equation}and
\begin{equation}\label{j21}t^{\frac{3}{2}\left(1-\frac{1}{p}\right)}\partial_i\partial_j X\in C_b([0,\infty);L^p),\end{equation}for $i,j=1,2,3.$

Then, similarly as in \cite[Eqs. (3.4)-(3.10)]{b1}, one may deduce as-well that

$$\begin{aligned}& t^{1-\frac{3}{2p}}X\in C_b([0,\infty);\ L^r(\Omega;L^\frac{3p}{3-p})),\ r\geq1,\\&
t^{\frac{3}{2}\left(1-\frac{1}{p}\right)}\partial_iX\in C_b([0,\infty);\ L^r(\Omega;L^p)),\\&
t^{\frac{3}{2}\left(1-\frac{1}{p}\right)}\partial_i\partial_jX\in C_b([0,\infty);\ L^r(\Omega;L^p)),\ \forall r\geq1, \ i,j=1,2,3.\end{aligned}$$

Finally, if in equation (\ref{j1}) one applies the operator $K$, we get for $X$ the equation
$$\begin{aligned}X(t)&=K(e^{t\Delta} \Gamma(t)\ curlx)+\int_0^t K\left(e^{(t-s)\Delta}\Gamma(t)\Gamma^{-1}(s)[K(\ curl X(s))\cdot \nabla](\ curl X(s))\right)ds\\&
-\int_0^t K\left(e^{(t-s)\Delta}\Gamma(t)(\Gamma^{-1}(s)(\ curl X(s))\cdot \nabla)(\ K(curl X(s)))\right)ds.\end{aligned}$$
The above equation may be viewed as the random version of the Navier-Stokes equation (\ref{e1}), the $3-D$ case. However, since $U_0$ is not $\mathcal{F}_0-$measurable, the process $t\rightarrow U(t)$ is not $\mathcal{F}_t-$adapted, and so $X$ is not $\mathcal{F}_t-$adapted, too. By Theorem \ref{t2} we know that the above equation has a unique solution.

\section{Existence of  solutions  to a stochastic Navier-Stokes equations, up to a stopping time, adapted to the Brownian motion}
Recall the notations and the results from Theorem \ref{t1}. For each $r>0$ we define the stopping time
$$\tau_r:=\inf\left\{t\geq0;\ \eta_t\geq r\right\}.$$ Then, $\tau_r$ goes to infinity for $r\rightarrow\infty$. If $U_0(\omega)\in L^\frac{1}{1-\gamma}$ is such that
$$\sup_{0\leq s\leq \tau_r(\omega)}\eta_s(\omega)|U_0(\omega)|_\frac{1}{1-\gamma}<C,$$then equation  (\ref{j7}) has a unique solution $y=y(t,\omega),\ t\in [0,\tau_r(\omega)],\ y(\omega,0)=U_0(\omega).$   Once we fix $r>0$, noticing that $\eta_s\leq r$  if $s\leq \tau_r$, we deduce that in the case that $U_0\in L^\frac{1}{1-\gamma}$ is deterministic and $|U_0|_\frac{1}{1-\gamma}\leq \frac{1}{r}C,$ we have
$$\sup_{0\leq s\leq \tau_r(\omega)}\eta_s(\omega)|U_0|_\frac{1}{1-\gamma}\leq C,\ \mathbb{P}-\text{a.e.} \ \omega \in \Omega.$$Now, define
$$y^{\tau_r}(t):=\left\{\begin{array}{l}y(t),\ t\in[0,\tau_r],\\
y(\tau_r),\ t\geq\tau_r.\end{array}\right.\ $$ Since $U_0$ is deterministic, it follows that $y^{\tau_r}$ is $(\mathcal{F}_t)_{t\geq0}-$adapted and so is $U^{\tau_r}(t)=\Gamma(t)y^{\tau_r}(t),\ t\geq0.$ By the stochastic calculus, we conclude that $U^{\tau_r}$ solves the stochastic vorticity equation  \eqref{e2000} on $[0,\tau_r]$, while $X^{\tau_r}(t)=K(U^{\tau_r}(t))$ solves the stochastic Navier-Stokes equation on $[0,\tau_r]$. These lead to the following corollary of Theorem \ref{t1}
\begin{corollary}For each $r$ and deterministic $U_0\in L^\frac{1}{1-\gamma}$ satisfying the condition
$$|U_0|_\frac{1}{1-\gamma}\leq \frac{C}{r},$$there is a unique solution $U=U(t,\omega)$ to the vorticity equation (\ref{e2000}) up to an explosion time $\tau_r$ adapted to the Brownian motion.\end{corollary}
A similar corollary holds true for the 3-D case  in Theorem \ref{t2}. As mentioned in \cite[Remark 4.2]{b1} this local existence and uniqueness result for the stochastic Navier-Stokes equation is new due to the fact of low regular initial condition requirements. Moreover, one can solve the Navier-Stokes equation in vorticity form on a nonempty time interval $[0,\tau_r)$ for deterministic $U_0\in L^\frac{1}{1-\gamma}$ ( $U_0\in L^\frac{3}{2}$ in the 3-D case).

\section*{Acknowledgement} I.M. was supported by a grant of the "Alexandru Ioan Cuza" University of Iasi, within
the Research Grants program, Grant UAIC, code GI-UAIC-2018-03. Financial support by the DFG through CRC 1283 is gratefully 
acknowledged by M.R.

\end{document}